\documentclass[11pt]{article}
\usepackage[leqno]{amsmath}
\usepackage{amssymb}
\usepackage{amsthm}
\usepackage{amsfonts}
\usepackage{graphicx}
\usepackage{enumerate}

\usepackage{setspace}

\newtheoremstyle{nonum}{}{}{\itshape}{}{\bfseries}{.}{ }{\thmnote{#3}}

\newtheorem{thm}{Theorem}[section]%[subsection]

\newtheorem{cor}[thm]{Corollary}
\newtheorem{lem}[thm]{Lemma}
\newtheorem{prop}[thm]{Proposition}
\newtheorem{rem}[thm]{Remark}

\newtheorem{definition}[thm]{Definition}

\theoremstyle{nonum}

\newcommand{\R}{\mathbb R}
\newcommand{\RR}{\mathbb R}

\def\L{{\cal L}}
\def\P{{\cal P}}

\newcommand{\iprod}[2]{\langle #1,#2 \rangle} % Inner product %

\def\calL{{\cal L}}
\def\L{{\calL}}

\def\ra{\rightarrow}
\def\del{\partial}
\def\Cvx{{\rm Cvx}}

\def\SCvx{{\rm SCvx}}

\def\fpolar{{f^\circ}}

\def\grad{{\nabla}}

\def\MA{{Monge--Amp\`ere }}

\def\Lf{{f^\star}}
\def\Pf{{f^\circ}}

\def\Jf{{f^{\circ\star}}}
\def\pa{\partial}
\def\epi{{\rm epi}\,}
\def\Im{{\rm Im}\,}
\def\dom{{\rm dom}\,}

\def\h#1{\hbox{#1}}
\def\psg{\del^\circ}
\def\pg{\nabla^\circ}
\def\eps{{\varepsilon}}
\def\dom{{\rm dom}}
\def\int{{\rm int}}

\begin{document}
 \title
 {Analysis
 of polarity}
 \author{Shiri Artstein-Avidan and Yanir A. Rubinstein }
 \maketitle
\date

\begin{abstract}
We develop a differential theory for the polarity transform
parallel to that of the Legendre transform,
which is applicable when the functions studied
are ``geometric convex'', namely convex,
non-negative and vanish at the origin.
This analysis may be used to
solve a family of first order equations reminiscent
of Hamilton--Jacobi and conservation law equations,
as well as  some second order \MA type equations. A special case of the latter,
that we refer to as the homogeneous polar Monge--Amp\`ere equation, 
gives rise to a canonical method of interpolating between convex functions.
\end{abstract}

\section{Introduction}

The Legendre transform $\L$, 
introduced by Mandelbrojt and Fenchel, is a classical
operation mapping functions on $\R^n$ to convex
lower-semi-continuous functions. It has numerous applications in
many areas of mathematics, in physics and in economics. Restricted
to convex lower-semi-continuous functions, it is an involution and
on twice differentiable convex function satisfies 
\begin{equation}
\label{LegendreIdentitiesEq}
\nabla\Lf =
(\nabla f)^{-1}, 
\h{\ \ and\ \ } 
\nabla^2 f|_x = (\nabla^2 \Lf)^{-1}|_{\nabla f(x)},
\end{equation} 
where we denote $\Lf:=\L f$. These properties lead to the
classical fact that $\L$ can be used to solve various first- and
second-order equations, in particular equations of conservation
laws, Hamilton--Jacobi equations, 
and \MA equations. 

Our main focus in this article is another duality transform $\P$,
called {\it polarity}.
Our main goal here is to develop a differential theory
for $\P$. We introduce the notion of a polar subdifferential for a function, and analyze its properties. The analysis turns out to be more delicate than the corresponding analysis for $\L$, due to the more non-linear nature of this transform. We further identify a wide class of convex functions for which second order analysis can be developed.   
As applications of this analysis
we are led to introduce certain PDEs
that are natural analogues of the classical Hamilton--Jacobi, conservations law and
Monge-Amp\`ere equations. These can solved  by the polarity transform. They provide new processes
for interpolation between convex functions.

Due to the ubiquitous role of the Legendre transform, the results here naturally raise the possibility of
deriving many other parallel constructions and applications for
polarity. The differential analysis of polarity we initiate here can be
seen as a first step in this direction.
Further generalizations, applications, and interpretation in terms of affine differential geometry will be considered
elsewhere.

This article is organized as follows. After deriving some basic
identities for polars of nonnegative functions in Section
\ref{PolarityNonNegSec},
Section \ref{PolaritySubDiffSec} is concerned with the
basic sub-differential theory for polarity.
Here the polar subgradient is defined and some of its
basic properties are studied. Section \ref{PolaritySecondOrderSec}
computes the Hessian under polarity. In Section \ref{PolarityVariationSec} we
compute the first- and second-variation formulas for families of polars.
Sections \ref{FirstOrderSec}--\ref{SecondOrderSec} derive the canonical
Hamilton--Jacobi and \MA type equations associated to polarity. In Section \ref{sec:Ginf} we derive some explicit formulas for these solutions.

\section{Polars of nonnegative functions}\label{PolarityNonNegSec}

Recently it was shown \cite{AM2} that the
Legendre--Mandelbrojt--Fenchel transform \cite{Fe,Ma}
\begin{equation}\label{eq:deflegendre}
 \Lf(y)\equiv (\L f)(y)  = \sup_{x \in \R^n}
\left(\langle x, y\rangle - f(x)\right), \end{equation}
and polarity \cite[\S15]{R}
\begin{equation}\label{eq:defpolaritytrans}
\Pf(y)\equiv (\P f)(y) = \inf\{ c\ge 0\,:\, \langle x,y\rangle\le 1+
cf(x), \forall x\in\RR^n\},
\end{equation}
are essentially the {\em only} order reversing involutions on the
class
$$
\Cvx_0(\R^n):=\{f \h{ convex and lower semi-continuous on $\RR^n$,\ } f\ge 0, \; f(0)=0\},
$$
referred to as the class of ``geometric convex functions."
We denote by
$\Cvx(\R^n)$ the set of lower semi-continuous convex functions
$f: \R^n \to \R \cup \{ + \infty\}$.
Note that functions in $\Cvx_0(\RR^n)$ are
always proper and closed in the terminology of Rockafellar
\cite{Ro}. The domain of a function in $\Cvx(\R^n)$ is defined to be the (convex)
set on which it attains finite values, and is denoted $\dom(f)$.
Let us remark that the notation in the present
article clashes somewhat with that in \cite{AM2}.

The epigraph of a function is defined as the set
\[ \epi f = \{
(x,y) \in \R^n \times \R : f(x) \le y\}, 
\] 
Note that a function $f$
belongs to $\Cvx_0(\R^n)$ if and only if the epigraph is a closed
convex set containing $\{0\} \times \R^+$ and
contained in the half-space $\R^n \times \R^+$.

\begin{lem}\label{lem:PropertiesofA}
Let $f$ be a non-negative function. (i) Then  $\Pf\in \Cvx_0(\RR^n)$
and
\begin{equation}
\label{GeomDualityDef}
\fpolar(y)
=
\begin{cases}
\sup_{ x\not\in f^{-1}(0) } \frac{\langle x,y\rangle -
1}{f(x)}, &\quad \hbox{for\ \ $0\not=y \in \{x\,:\, f(x) =
0\}^{\circ}$}, \cr 0&\quad y=0, \cr +\infty &  \quad \h{otherwise. }
\end{cases}
\end{equation}
(ii) The double polar of $f$ is the convex 
envelope,
\begin{equation}\label
{DoubleDualEq}
f^{\circ\circ}=f^{\star\star}=\sup\{g\in\Cvx_0(\RR^n)\,:\, g\le f\}\le f.
\end{equation}
(iii) The epigraph of $\Pf$ is the reflection with respect to $\R^n
\times \{0\} \subset \R^{n+1}$ of the polar of  the epigraph of $f$.
\end{lem}

\begin{proof}
If $f\equiv0$ then both
\eqref{eq:defpolaritytrans} and \eqref{GeomDualityDef} give $\Pf=1^c_{\{0\}}$.
So assume that $f\not\equiv 0$.\\
(i) First observe that by \eqref{eq:defpolaritytrans}, $\Pf(0)=0$.
If $f\in\Cvx_0(\RR^n)$, the first line of \eqref{GeomDualityDef} already
implies $\Pf(0)=0$ since $f$ is unbounded.
For general $f\ge 0$,
\[
\sup_{ x\not\in f^{-1}(0) }\frac{\langle
x,y\rangle - 1}{f(x)} = \inf \{ c:  {\langle x,y\rangle }\le 1 +
c{f(x)}\,\, \forall x\not\in f^{-1}(0)\}.
\]
If $0\not=y\in (f^{-1}(0))^{\circ}$ then for all $x\in f^{-1}(0)$ we have that $
{\langle x,y\rangle }\le 1$ and
\[
\begin{aligned}
\sup_{ x\not\in f^{-1}(0) }\frac{\langle
x,y\rangle - 1}{f(x)}
& =
\inf \{ c:  {\langle x,y\rangle }\le 1 + c{f(x)}\,\, \forall x \in \R^n\}
\cr
& =
\inf \{ c\ge0:  {\langle x,y\rangle }\le 1 + c{f(x)}\,\, \forall x \in \R^n\},
\end{aligned}
\]
since for some $x$ with ${\langle x,y\rangle }> 1$ we have $f(x)>0$
(as $f \not\equiv 0$).
If $y\not\in (f^{-1}(0))^{\circ}$,
then there exists some $x$ with $f(x) = 0$ and with ${\langle
x,y\rangle }> 1$, thus by
(\ref{eq:defpolaritytrans}), $\Pf(y)=
+\infty$,
in agreement with (\ref{GeomDualityDef}).
Finally, to see that $\Pf\in \Cvx_0(\RR^n)$, it only remains to show
that it is convex. If $f$ is unbounded, then as already noted
the first line of \eqref{GeomDualityDef} already implies that $\Pf(0)=0$,
and then $\Pf$ is a supremum of linear functionals and
$1^c_{(f^{-1}(0))^\circ}$, hence convex.
Finally, if $f$ is bounded then $\Pf=1^c_{\{0\}}$.
\\
(ii) By the classical properties of the Legendre transform \cite{Ro}
it suffices to show the first
equality in \eqref{DoubleDualEq}. First, if $f\in\Cvx_0(\RR^n)$,
$f^{\star\star}=f^{\circ\circ}=\h{cl}\, f = f$, where $\epi\h{cl}\,
f=\overline{\epi f}$ \cite{Ro}. Observe now that $f^{\circ\circ}\le
f$, indeed:
$$
\begin{aligned}
f^{\circ\circ}(x)
& =
\inf\{c^\star\ge0\,:\, \langle x ,z\rangle\le 1
            +c^\star
\Pf(z), \;\forall z\}
\cr
& \le
\inf\{c^\star\ge0\,:\, \langle x,z\rangle\le 1
            +c^\star\frac{\langle x,z\rangle-1}{f( x)}, \;\forall z\}
= f(x).
\end{aligned}
$$
Since clearly $\P$ is order reversing, we see that $\h{cl}\, f \le f$ implies 
$(\h{cl}\, f)^{\circ} \ge f^{\circ}$ and 
$\h{cl}\, f = (\h{cl}\, f)^{\circ \circ} \le f^{\circ\circ} \le f$. 
However,  $f^{\circ\circ}$ is closed by (i) and thus must equal $\h{cl}\, f$. 
\\
(iii)
The statement holds for $f\in\Cvx_0(\RR^n)$ essentially from
the definition \eqref{eq:defpolaritytrans} \cite[p. 137]{Ro}.
In general, by  (i) $\Pf\in\Cvx_0(\RR^n)$, and so by (ii) $f^{\circ\circ\circ}=f^\circ$,
implying the desired statement.
\end{proof}

The previous lemma
recovers well-known properties of polars of non-convex sets. 
Let $K$ be a set in $\RR^n$, and let $K^{\circ}$ denote its polar, 
given by
$$
K^{\circ} = \{ y \in \R^n : \langle x, y \rangle \le 1, \;\; \forall
x \in K\}.
$$
For a closed convex set $K$ let $1^c_K$  
denote
the convex indicator function, equal to $0$ on $K$ and
$+\infty$ elsewhere. 
Then $\P 1^c_K =
1^c_{K^{\circ}}$.

One more useful fact is that for $f\in \Cvx_0(\R^n)$ we have that 
\begin{equation}\label{eq:domdual} \overline{\dom (\Pf)} = \{x:  f(x) = 0\}^{\circ}
{\rm ~~~and~~~} 
 \overline{\dom (f)} =    \{y: \Pf (y) = 0\}^{\circ}.\end{equation}
Indeed, a closed convex set $K$ satisfies that $K = \overline{\dom(f)}$ if and only if $f\ge 1^c_K$ and $f\not\ge 1^c_{K'}$ for any closed $K'\supsetneq K$. 
Similarly, $\{ f = 0\} = T$ if and only if $f\le 1^c_T$ and $f\not \le 1^c_{T'}$ for any $T'\subsetneq T$. Since polarity on $\Cvx_0(\RR^n)$ is an involution which changes order and replaces $1^c_K$ by $1^c_{K^{\circ}}$, the claim follows. 

Next we briefly discuss the 
composition of $\P$ and $\L$. It is not hard to check that the two transformations
commute, and thus the composition is an involution on $\Cvx_0(\R^n)$
which is order preserving. We list two of its properties which shall be useful in the sequel.

\begin{lem}\label{lem:JDualEq}
Let $f\in \Cvx_0(\R^n)$ and $x$ with $f(x) \neq 0, +\infty$.
If $f$ is differentiable at the origin, we have that
\begin{equation}\label{JDualEq} f(x)\Jf(x/f(x))=1.
\end{equation}
Moreover, the above conclusion holds whenever $f|_{[0,x(1+\delta)]}$
is not linear for any $\delta>0$.
\end{lem}

\begin{proof}
We will prove this lemma using mainly the order-preservation property of $\L\P$,
together with our knowledge on how it acts on simple functions. Indeed, by the properties above
it is enough to consider
functions in $\Cvx_0(\R^+)$. Clearly $f \le \max (1^c_{[0,x]},
l_{\frac{f(x)}{x}})$, where $l_c$ denotes the function $l_c(t) =
ct$. Since $(1^c_{[0,x]})^{\circ\star} = l_{1/x}$ and $l_c^{\circ\star}
= 1^c_{[0,1/c]}$ we get that
\[ f^{\star\circ} \le \max (1^c_{[0,x/f(x)]}, l_{1/x}),\]
so that $f^{\star\circ}(x/f(x)) \le 1/f(x)$. Next, we use the
assumption that $f$ has a supporting functional at $x$ which is {\em
not} the linear function $l_{x/f(x)}$. Denote this support function
by $\frac{f(x)}{x-w}(t-w)$, and let $h(t) = \max(0,
\frac{f(x)}{x-w}(t-w))\le f$. Then $\Jf\ge h^{\star\circ}$ which is
easily computed to be $\max(0, \frac{1}{w}(t - \frac{x-w}{f(x)}))$.
In particular we have $f^{\star\circ}(x/f(x)) \ge
h^{\star\circ}(x/f(x)) = 1/f(x)$. The proof of the lemma is
complete. \end{proof}

\begin{rem}
{\rm
 In the  case not covered in
Lemma \ref{lem:JDualEq},  the product in (\ref{JDualEq}) can still be computed.
Indeed, a function $f\in \Cvx_0(\R^n)$ is linear on some interval $[0,y]$, if and only if
the mapping $x\mapsto x/f(x)$ is not injective. Assume
that $f(ty) = ct$ for $0\le t\le 1$, and that $f$ is not linear on any extension
of $[0,y]$. Then $\Jf|_{\R^+y}$ is supported on $[0,y/f(y)]$,
and the value it attains on $y/f(y) = ty/f(ty)$ is $1/f(y)$. Thus
for any $0<t\le 1$,
\[ f(ty)\Jf(ty/f(ty))=t.\]
}
\end{rem}

The second propoerty of $\Jf$ which we shall need is a geometric description, which will help us investigate how properties such as smoothness and strict convexity are affected by this transformation $f\mapsto \Jf$. 
Define the mapping $F: \RR^n \times (0, \infty)  \to \RR^n \times (0, \infty)$ by 
\[F(x_1, \ldots, x_n, t) = (x_1/t, \ldots, x_n/y, 1/t).\]
The following was shown in \cite{AM2}:
\begin{lem}\label{lem:JisF}
Let $f\in Cvx_0(\RR^n)$, then  
\[{\rm int}({\epi} (\Jf )) = F({\rm int}({\epi} (f))).\]
\end{lem}

\begin{rem}{\rm
The mapping $F$ is ``fractional linear''(sometimes called ``projective linear''), and in particular maps segments to segments and subdomains of affine $k$-dimensional subspaces to subdomains of affine $k$-dimensional subspaces. In particular, if $f\in Cvx_0(\RR^n)$ is strictly convex outside of $\{f= 0\}$ then $\Jf$ is strictly convex outside $\{\Jf = 0\}$, and if $f$ is differentiable in $\dom(f)\setminus \{f= 0\}$ then $\Jf$ is differentiable in $\dom(\Jf) \setminus \{\Jf = 0\}$.  \\
Note, however, that even for $f$ which is everywhere differentiable and strictly smooth, the function $\Jf$ may have a non-zero set $\{\Jf = 0\}$ and may fail to be differentiable at the boundary of this set.
}
\end{rem}

The following definition will be of use to us in the sequel. 
\begin{definition}
Let $f\in \Cvx_0(\RR^n)$.\\ 
(i) The {\rm ray-linearity zone } of $f$ is the set 
of $x\in \RR^n\setminus\{0\}$ such that $f|_{[0,x]}$ is linear.\\
(ii) {\rm $f$ is nonlinear at infinity } if for every ray $\RR^+x$ the one dimensional function $f(tx)$ defined on $t\in{\RR^+}$ has domain $\RR^+$ and is not between $h(t)=at$ and $h(t)-b$ for any $a\ge 0$ and $b>0$. 
\end{definition}

\begin{lem}\label{lem:linzone}
A function $f\in \Cvx_0(\RR^n)$ is nonlinear at infinity if any only if the linearity zone of $\Pf$ is empty. 
\end{lem}

\begin{proof}
Put $g:=\Pf$. A function $g$ in $\Cvx_0(\RR^n)$ is linear on a segment $[0,w]$ if and only if it is above some linear function $l(x) = \iprod{x}{u}_+$ ($u$ not orthogonal to $w$) and below the function which in $+\infty$ everywhere and equal to $\iprod{x}{u}$ on the segment $[0,w]$. This function may be written as $\sup (l, 1^c_{[0,w]})$. By taking polars of these conditions, we get that this happens if and only if $g^{\circ}$ is below $l^\circ (x)$ and above $\hat{\inf} (l^\circ, 1^c_{[0,w]^{\circ}})$. Note that $l$ is simply the norm associated with a certain body (a halfspace in fact) so that $l^\circ$ is simply the norm associated with the polar of this body, which is the segment $[0,u]$. This norm is infinity outside $\RR^+u$ and equals to $\iprod{x}{\frac{u}{|u|^2}}$ on this ray.   As for $\hat{\inf} (l^\circ, 1^c_{[0,w]^{\circ}})$, it is the convexified minimum of $l^\circ$ and of the indicator of a halfspace, so it is $0$ on the halfspace and linear outside, with slope the same as $l^\circ$ was. By Hahn Banach theorem, this is equivalent to the fact that $g^{\circ}(tw)$ restricted to $t\in \RR^+$ is below the linear function $h(t)= t\iprod{w}{\frac{u}{|u|^2}}$ and above $h(t)-b$ for some $b>0$.
\end{proof}

\section{Polar subdifferential map }\label{PolaritySubDiffSec}

A central feature of the Legendre transform is that it is related to
a gradient mapping. 
Namely, when $f\in C^1(\R^n)$ is strictly convex,
$\nabla f:\RR^n\ra(\RR^n)^\star\cong\RR^n$ is injective, and
$$
\Lf(y):=\sup_{x\in\RR^n}[\langle x,y\rangle - f(x)]
$$
can be computed explicitly from the function and its gradient map,
\begin{equation}
\label
{LegendreExactFormulaEq}
\Lf (y)=\langle (\nabla f)^{-1}(y),y\rangle - f (\nabla
f)^{-1}(y)).
\end{equation}
More generally, for any proper closed convex function $f$, one uses the subdifferential map $\del f (x) = \{ y:
f(z) \ge f(x) + \langle  z-x, y\rangle \,\,\, \forall z\},$ and the
inverse of the subdifferential map detects the points where the
supremum is attained, so that $\del f (x) = \{ y: \Lf (y) + f(x) =
\langle  x,y\rangle\}$.  Moreover,
$\del f=(\del \Lf)^{-1}$, i.e., $ y \in \del f(x) {\rm
~if~and~only~if~} x\in \del \Lf(y). $
The above facts motivate the following definition for the polar-subdifferetial map.

\begin{definition}
For a function $f:\R^n \to \R^+\cup \{\infty\}$ define the
polar subdifferential map $\del^\circ f$ at a point $x\in \dom (f)$ by
\begin{equation}\label{eq-psg}
\psg f(x):= \left\{ y\in \dom(f^{\circ}): \Pf(y)f(x) = \langle x,y\rangle-1
\right\}.
\end{equation}
We say that $y\in\psg f(x)$ is a polar subgradient of $f$ at $x$.
The {\em domain} of $\psg f$, $\dom (\psg f)$, is defined as the set of $x$ with $\psg f (x) \neq \emptyset$.
\end{definition}
Note that $\psg f (x)$ is a  convex set.  Indeed, $\Pf ((1-\lambda)y_1+\lambda y_2)\le (1-\lambda)\Pf (y_1)+\lambda \Pf(y_2)$ hence if $y_1,y_2\in \psg f (x)$ then $\Pf((1-\lambda)y_1+\lambda y_2)f(x) \le \langle x,(1-\lambda)y_1+\lambda y_2\rangle-1$, and as the opposite inequality is always true by definition of $\Pf$, we have equality on $[y_1, y_2]$. Also note that $\psg f (x)$ is closed relatively to $\dom (\Pf)$. 

\begin{rem}\label{rem-zerosets}
{\rm One could roughly restate the definition above in words as follows: ``$y$ is a polar sub-gradient of $f$ at $x$ if the supremum in the definition of $\Pf(y)$ is attained at $x$.''  The case for which this second definition does differ from the one above is when $f(x) = 0$. Let us shortly discuss this case:
First note that $\psg f (0) = \emptyset$. Consider some $x\neq 0$ with $f(x)=0$. 
In such a case $ \psg f(x)= \left\{ y\in \dom(f^{\circ}):
\langle x,y\rangle=1 \right\}$. If the function $f\in Cvx_0(\RR^n)$ has zero set
$K$, that is, $f|_{K} = 0$  and  $f|_{\R^n\setminus K} \neq 0$, then
$\overline{\dom(\Pf)} = K^{\circ}$ so that the above definition
becomes, for $x\in K$,
\[
\psg f(x):= \left\{ y\in K^{\circ} :  \langle x,y\rangle=1,
\Pf(y)\neq \infty \right\}.
\] For $x$ in the relative interior of $K$ this is again an empty set, and
for $x$ on the boundary of $K$ the polar subdifferential is the set of supporting
functionals to $K$ at $x$ (which are in the boundary of $K^\circ$) and which belong to $\dom(\Pf)$. 
For example we may consider a function $\Pf$ with $\dom(f) = \int(K^{\circ})$, so that $f$ itself is $0$ on $K$ and for all $x$ in the boundary of $K$, the polar-subdifferential is empty.} 
\end{rem}

Note that if $f\in\Cvx_0(\R^n)$ then by \eqref{DoubleDualEq},
$(\del^\circ f)^{-1}=\psg\Pf$, i.e., $x\in \psg\Pf(y)$ if and only if
$y\in \psg f(x)$. This means that one may write 
\[ \psg \Pf (y) = \{ x\in \dom (f) : y\in \psg f (x)\}, \]
 that is, the polar subgradients of $\Pf$ at a point $y$ with $\Pf (y) \neq 0$ are precisely the points for which the supremum in the definition of $\Pf$ is attained. This allows us to examine many examples for which $\psg f$ is empty, for example when $f$ is a norm then clearly in the definition of $\Pf$ (which is the dual norm) the supremum is never attained.

Note that unlike the usual subdifferential,
$\psg f(x)$ can be empty even when $f$ is smooth and convex at $x$.
This and other properties of $\psg$ will follow from the following
basic lemmas. Below we say that 
``$g|_{[0,x+]}$ in not linear'' when there is no 
interval $[0,tx]$ with $t>1$ on which $g$ is linear.

\begin{lem}\label{lem:moment-map-gencase} Let $f\in \Cvx_0(\R^n)$. Then for each
$x\in \int(\dom(f))$ with $f(x)\not=0$,
\begin{eqnarray*}
 \psg f(x) &=&  \{ y: \Pf(y), \Lf(y/\Pf(y))
 \neq 0,+\infty, \Pf|_{[0,y+]}{\rm
 ~not~linear~and~}
  y/\Pf(y) \in \partial f (x) 
\} \\&=& \left\{z/\Lf(z)\,:\, z\in\del f(x), \Lf|_{[0,z+]} {\rm ~not~linear~and~} \Lf(z), \Pf (z/\Lf(z))\neq 0, +\infty %, ty\not\in \del
\right\}.
\end{eqnarray*}
If, in addition, $f$ is not linear on $[0,x+]$ and $f(x), \Jf(x/f(x))\neq 0, +\infty$ we have that 
\begin{eqnarray*}
\psg f(x) &=& (\del\Pf)^{-1}(x/f(x)) := \{ y\,:\, x/f(x) \in
\partial \Pf (y) \}.
\end{eqnarray*}
\end{lem}

Note that the condition $x\in \int(\dom(f))$ is important, for instance check the example 
$f(x) = 
\hat{\inf} \left( 1^c_{[0,1]}, \max (x/2, 1^c_{[0,2]})\right)$ at $x = 2$. Here $\hat{f}$ denotes the closed convex envelope of $f$. In this example $\psg f (2) = [1/2, 1]$ whereas the expression in the right hand side of the first equation in the lemma gives $(1/2, 1]$.  

\begin{proof}
First, $y\in\psg f(x)$ satisfies $\Pf (y) \neq 0$ otherwise by Remark \ref{rem-zerosets}
$x$ would be at the boundary of the domain of $f$, contrary to the assumption.
Thus, in  equation \eqref{eq-psg} one may divide by $\Pf(y)$ so that the condition
 is equivalent to $\langle  \cdot ,
y/\Pf(y)\rangle-\frac1{\Pf(y)}$ defining a supporting hyperplane for
$f$ at $x$. In particular, $y/\Pf(y) \in \partial f(x)$. This in turn implies (by Legendre theory) that
$f(x) = \langle x,
y/\Pf(y)\rangle - {\Lf(\frac{y}{\Pf(y)})}$ and comparing this with the definition of $\psg f (x)$ we get that $\Lf(\frac{y}{\Pf(y)}) = \frac{1}{\Pf(y)}$ so that in particular $\Lf (\frac{y}{\Pf(y)})\neq 0,+\infty$.
Finally, were it the case that $\Pf$ was linear on some interval $[0,yt]$, for some $t>1$ then we'd have that
\[ \Pf(ty) = t\Pf(y) = \frac{\langle x,ty\rangle - t}{f(x)}<
\frac{\langle x,t y\rangle - 1}{f(x)}\le \Pf(ty), \] a contradiction.

On the other hand, suppose that  $\Pf(y),
\Lf(\frac{y}{\Pf(y)})
\neq 0, +\infty$, $\Pf|_{[0,y+]}$ not linear and
$y/\Pf(y) \in \del f(x)$.
Then $\langle \cdot , y/\Pf(y)\rangle - {\Lf(\frac{y}{\Pf(y)})}$ is
a supporting hyperplane for $f$ at $x$, and in particular $f(x) =
\langle x, y/\Pf(y)\rangle - {\Lf(\frac{y}{\Pf(y)})}$. 
By Lemma
\ref{lem:JDualEq}, applied to $\Pf$ (which can be applied since $\Pf$ is not linear on
$[0,y+]$)
we have that 
${\Lf(\frac{y}{\Pf(y)})} = \frac{1}{\Pf (y)}$ which  concludes
the proof of second inclusion and thus completes the first equality. 

For the second equality, note that by the above argument, letting $z
= y/\Pf(y)$ we have that $y = z/\Lf (z)$.  Thus, if $y\in \psg f(x)$
then it can be written as $z/\Lf (z)$ for some $z$ satisfying the
conditions. (Indeed, one only should notice that linearity of $\Pf$ on $[0,w]$ is equivalent to linearity of $\Lf= (\Pf)^{\star\circ}$ on $[0, w/\Pf(w)]$, by the properties of the transform $\P\L$ discussed in Section \ref{PolarityNonNegSec}). On the other hand, given $y= z/\Lf(z)$ in $dom(\Pf)$
such that $\Pf(y) \neq 0$ and $\Pf$ is not linear on $[0,y+]$, we
see again, that $y/\Pf(y) = z$, otherwise $\Pf(z/\Lf(z))\Lf(z) \neq
1$ which implies that $\Lf$ is linear on $[0,z+]$, a contradiction to
the assumptions.
This completes the proof of the first assertion.

For the second, simply use that $y\in
\psg f (x)$ if and only if $x\in \psg \Pf (y)$ together with the first
assertion.
\end{proof}

An easy consequence is the following:
\begin{cor}\label{SecondPgCor}
Let $f\in\Cvx_0(\RR^n)$  and $x\in \int(\dom(f))$ with $f(x)\neq 0$, and assume $f$ is not linear on $[0,x+]$. Then
\begin{equation}\label{PgCharEq}
\psg f(x)
=
\del\Jf(x/f(x)),\end{equation}
and
\begin{equation}
\psg \Jf(x/f(x))
=
\del f(x).
\end{equation}
\end{cor}

\begin{proof}
By Lemma
\ref{lem:moment-map-gencase}, for $f$ satisfying  these conditions,
 $y\in\psg f(x)$ is equivalent to
$x/f(x)\in\del\Pf(y)$, which in turn, by Legendre theory, is
equivalent to $y\in\del\Jf(x/f(x))$. The second equation follows after noticing that $f$ is linear on $[0,w]$ 
if and only if $\Jf$ is linear on $[0,w/f(w)]$, and then applying the first equality to $\Jf$ at $x/f(x)$, noticing that 
under non-linearity of $f$ on $[0,x+]$ we have $(x/f(x))/\Jf(x/f(x))=x$ 
by \eqref{JDualEq}.
\end{proof}

Differentiability and polar differentiability are related by the next lemma.

\begin{lem}\label{SecondPgLemma}
If $f\in\Cvx_0(\RR^n)$ is differentiable and nonzero at $x\in \int(\dom(f))$
then
either $\psg f(x)\!=\!\emptyset$
or $\psg f(x)\!=\!\{y\}$ where
\begin{equation}\label{eq:FpolarRepresenation}
y=\frac {\grad f (x)}{\Lf(\grad f (x))},
\qquad \h{and } \qquad
\fpolar(y)=\frac1{\Lf(\nabla f(x))}=\frac
1{\langle x,\nabla f(x)\rangle-f(x)}.
\end{equation}
Conversely,
if $f\in\Cvx_0(\RR^n)$, $x \in \int(\dom f))$ with $f(x)\neq0$, satisfy $\psg f(x)=\{y\}$   
then $\del f(x)=\{y/\Pf(y)\}$. In particular, $f$ is differentiable at $x$,
and $y=\nabla f(x)/\Lf(\nabla f(x))$.
\end{lem}

\begin{proof}
The first part follows directly from Lemma \ref{lem:moment-map-gencase}. Indeed, 
from differentiability there exists only one $y\in \partial f (x)$ so that the set 
$\psg f (x)$ can include at most one element, and in case it includes an element, this element must be $y/\Lf (y)$. 
If indeed $\psg f(x)$  includes one element, $y$, then, still from 
Lemma \ref{lem:moment-map-gencase}, $y/\Pf(y) = \nabla f (x)$ and
$\Pf$ is not linear on $[0,y+]$ so that $\Pf(y) = 1/\Lf(y/\Pf(y))= 1/\Lf(\nabla f(x))$. This completes the proof of the first part.

Suppose now that $\psg f(x)=\{y\}$ for some $x\in \int(\dom(f))$ with $f(x)\neq 0$.
By Lemma \ref{lem:moment-map-gencase} we have that $y \in \dom(\Pf)$,  $\Pf (y) \neq 0$, 
$\Lf (\frac{y}{\Pf(y)})\neq 0$ and $\Pf|_{[0,y+]}$ is not linear. Letting $z:= y/\Pf (y) $ we have that  $z\in \partial f (x)$. 
By  \eqref{JDualEq} (and the remark following it) we thus have that $\Lf$  is not linear on $[0,z+]$ and $\Pf(y) = \frac{1}{\Lf(z)}$. In particular, $y = \frac{z}{\Lf(z)}$.

Assume that there was another element $z'\in \partial f (x)$. 
From convexity of the set $\partial f (x)$ we can clearly assume that $z'$ is 
as close as we wish to $z$. 

We claim that for $z'$ close enough to $z$ we have that $\Lf$ is not linear on $[0,z'+]$. 
If indeed this is true, we can make sure by continuity of $\Lf$ on its domain that $y = z/\Lf(z)$ and $y':= z'/\Lf(z')$ are close and thus also $\Pf (y') \neq 0$, and by Lemma \ref{lem:JDualEq} also $\Pf (y') \neq +\infty$ so that $y'\in \psg f(x)$. We thus must have that $y' = y$ but this means that $\Lf$ is linear on $[0,z]$ and $z' = tz$ for some $t<1$ which is not the case we are considering. 

The remaining case is that we 
cannot find $z'\in \partial f (x)$ close to $z$ such that $\Lf$ is not linear on $[0,z'+]$. 
This means that $\Lf$ {\em is} linear on $[0,z]$ and the only other $z'\in \partial f (x)$ are $z' = tz$ for $t<1$. Since $z'\in \partial f(x)$ if and only if $x\in \partial \Pf(z)$ this linearity implies that $\partial f(x) = [0,z]$. This already implies that $f(x) = 0$ 
(since we have for any $0\le t\le 1$ in particular that  $f(0) \ge f(x) - \iprod{tz}{x}$) which contradicts the assumption on $x$. 
\end{proof}

In the case where $\psg f $ is a single point $\{y\}$ we say that $f$ is 
{\it polar differentiable at $x$} and denote the polar gradient by
$$\pg f(x)=y.$$

Some further consequences of Lemmas \ref{lem:moment-map-gencase} and
\ref{SecondPgLemma} are the following. First, as already remarked in Remark \ref{rem-zerosets}
above, $\psg f(x)=\emptyset$ if $x\in\h{int}\,f^{-1}(0)$. Second, if
$f$ is differentiable at $x\neq 0$ in the boundary of $f^{-1}(0)$ (and hence $\nabla
f(x)=0$) we also have $\psg f(x)=\emptyset$. Indeed, by the same
remark, were the equation in the definition to hold, we would need
$y$ to belong to the boundary of $(f^{-1}(0))^{\circ}$, and $\iprod{x}{y} = 1$. However, for
such point we have that $\Pf (y) = +\infty$ since by definition
\[ \Pf (y) \ge  \frac{\langle x(1+\eps), y\rangle -
1}{f(x(1+\eps))} =   \frac{\eps{\langle x , y\rangle}}{o(\eps)} \to
+\infty.  \]

Finally, there is a close connection between smoothness of $f$ and
differentiability of $\Pf$, similar  to the one holding for
Legendre transform.

Our main concern in sections below will be that if a function is  
 both strongly convex and twice continuously 
 differentiable, then so is $\Pf$. 
Most of this claim is proved in Section \ref{PolaritySecondOrderSec} where we 
derive a precise formula for the Hessian of $\Pf$. We shall need, however, a simpler claim regarding differentiability. To this end, we introduce the following class of functions:

\begin{definition}
Denote by ${\cal S}_1(\R^n)$ the class of $f\in \Cvx_0(\R^n)$ which
attain only finite values, are continuously differentiable on
$\R^n \setminus \{0\}$ and are strictly convex (that is, their graph does not contain any segment). 
\end{definition}
Note that these functions vanish only at the origin. 

\begin{prop}\label{prop:Pfisgood}
Assume $f\in {\cal S}_1(\RR^n)$.
Then \\
1. $\dom(\Pf)=\RR^n$, $\Pf$ vanishes only at the origin.\\
2. $f$ is polar differentiable  at any point $x\neq 0$.\\
3. $\Pf$ is  differentiable at any point $y\neq 0$ such that $\Pf_{[0,y]}$ is not linear. \\
4. $\Pf$ is strictly convex at any point $y$ such that $\Pf_{[0,y]}$ is not linear.  
\end{prop}

\begin{proof} 
Clearly $\dom (\Pf) = \{f= 0\}^{\circ} = \R^n$ and that $\{\Pf=\} = \dom(f)^{\circ} = \{0\}$. 

If $f$ is strictly convex and differentiable at
$x\in {\rm int}(\dom(f))$, then $f$ is polar differentiable at $x$. 
Indeed, by strict convexity of $f$ it cannot linear on $[0,x]$, which means $\Lf(\nabla f(x))\neq 0$. Thus by  Lemma 
\ref{lem:moment-map-gencase} the set $\psg f (x)$ is non empty, whereas by Lemma \ref{SecondPgLemma} 
it consists of at most one point. 

To show differentiability of $\Pf$, we consider the intermediate function $\Jf$.
By Lemma \ref{lem:JisF} 
\[{\rm int}({\epi} (\Jf )) = F({\rm int}({\epi} (f))).\]
and in particular (as $F$ is continuous) there can be no segments on the graph of $\Jf$ outside the set $\{(x,0): \Jf(x)=0\}$. This means that $\Jf$ is strictly convex at any point $x$ with $\Jf(x)\neq 0$. 
Therefore 
$\Pf = (\Jf)^{\star}$ is differentiable at any point $y$ such that $y=\nabla \Jf (x)$ for some $x$ with $\Jf(x)\neq 0$. This amounts precisely to $\Pf$ not being linear on $[0,y]$.

To get strict convexity of $\Pf$
we use that $\Jf$ is differentiable at any point with $\Jf(x)\neq 0$. Indeed, this follows from the remark after Lemma \ref{lem:JisF}, as any supporting $(n-1)$ dimensional region of $f$ is mapped via $F$ to 
a supporting $(n-1)$ dimensional region of $\Jf$ and vice versa. This means there is precisely one subgradient to $\Jf$ at any such point, and by Legendre theory there are no two points $y_1\neq y_2$ such that $Pf|_{[0,y_i]}$ is non-linear, in which $\Pf$ has the same gradient. That is, outside the ``ray-linearity zone'' of $\Pf$, it is strictly convex.    
\end{proof}

\section{Second order differentiability}\label{PolaritySecondOrderSec}

In this section we explain the relation between the Hessian of $f$ and the Hessian of $\Pf$. We shall mainly work in the following class of functions. 
\begin{definition}
Denote by ${\cal S}_2(\R^n)\subset  {\cal S}_1(\R^n)$ the class of  
twice continuously differentiable in $\RR^n \setminus \{0\}$ such that $\nabla^2 f(x)>0$ for all $x\neq 0$. 
\end{definition}

By Proposition \ref{prop:Pfisgood}   such functions are polar differentiable in $x\neq 0$. 
  In the following proposition we   derive a precise formula for the Hessian of $\Pf$ in terms of $\nabla^2 f$, at those points for which one can be sure $\Pf$ is twice differentiable.
\footnote{After presenting
the results from this article in the Cortona Convex Geometry Conference in
June 2011, we were informed by X.-N. Ma that
\eqref{HessianDetPolarityEq} was also obtained independently
by H.-Y. Jian, X.-J. Wang.
}
\footnote{We regard vectors $x\in \dom(f)$ as a column vector, $y\in \dom(\Pf)$ as row vectors, and the various differentials accordingly. For example,  the differential of $\Pf$, which is a function of $y$, is a matrix that is to be multiplied with vectors $v\in T_y(\RR^n)^\star$ from the right. 
When taking the differential of a function 
$G: X\to Y$ where points of $X$ are considered as column vectors and points of $Y$ as row vectors, we 
let $DG(x)$ act on $w\in T_x\RR^n$ by $(DG w)^T$, and similarly if $H: Y \to X$ 
(e.g., the kind of map $\nabla^{\circ}\Pf$ is) we let $DH$ act 
on a vector $v$ by $(vDH)^T$.}

\begin{prop}
\label{HessianPolarityProp} Assume $f\in {\cal S}_2(\R^n)$. Let $x\in \R^n\setminus \{0\}$ and  $y  =  \nabla^{\circ} f(x)$.  If $\Pf|_{[0,y]}$ is not linear then $\Pf$ is twice differentiable at $y$, $ \nabla^{2} \Pf(y)>0$ and we have
\begin{equation}
\label{HessianGeneralPolarityEq}
\left(f(x)\fpolar(y)\right)^2(\nabla^2\fpolar)(y) =
(f(x)I-x \nabla f(x))(\nabla^2 f(x))^{-1}(\fpolar(y) I- 
(\nabla\fpolar(y) y)^T),
\end{equation}
and
\begin{equation}
\label{eq:takedetlater}
(\nabla^2 f^{\circ}(y))^{-1}
=
{f(x)}f^{\circ}(y){\left( I -     {x}{y} \right)^T }
\nabla^2 f(x)
\left(I - {x}{y} \right).
\end{equation}
In particular,
\begin{equation}
\label{HessianDetPolarityEq} \det(\nabla^2 f (x)) \det (\nabla^2 \fpolar
(y)) = \frac{1}{(f(x)\fpolar(y))^{n+2}}.
\end{equation}
\end{prop}

\begin{proof}
By equation \eqref{eq:domdual}, the domain of $\Pf$ is $\R^n$ and it vanishes only at $0$. By Proposition \ref{prop:Pfisgood} the function $f^{\circ}$ is
differentiable at $y$. Thus $\nabla^{\circ} \Pf (y) = x$. 
By Lemma \ref{SecondPgLemma},  
\begin{equation}\label{eq-diffpfofx} \nabla \Pf (y) = \frac{x}{f(x)}, \quad{\rm and}\quad \grad f(x) = \frac{y}{f^{\circ}(y)}.
\end{equation}
The second equation implies
  $x = \nabla^{\circ}\Pf(y)$ is differentiable with
respect to $y$ and then the first that  
$\Pf$ is 
 twice differentiable.

Differentiating the second identity of \eqref{eq-diffpfofx} gives
\[ \nabla \nabla^{\circ}\Pf(y)  (\nabla^2 f)(x)  = \frac{1}{f^{\circ}(y)^2}\left[
f^{\circ}(y)I -  \grad f^{\circ}(y) y \right], \] 
where we denoted 
the differential of the map $x(y) = \nabla^{\circ}\Pf (y)$ by $\nabla \nabla^{\circ}\Pf(y)$.
Recall that for $\iprod{w}{z}\neq 1$ one has 
\[ \left( I - zw^T\right)^{-1} = I + \frac{zw^T}{1-\iprod{w}{z}},\]  
so that as $ \grad \Pf(y)/\Pf(y)= x$ and $\iprod{x}{y}\neq 1$  
(that is implied by $f(x)\neq 0$), 
$f^{\circ}(y)I - \grad f^{\circ} (y)y$ is invertible. 
As $f$ is strongly convex it follows that 
$\nabla \nabla^{\circ}\Pf(y)$ is positive-definite. 
As $ \nabla^{\circ}f \circ \nabla^{\circ}\Pf  =
Id$ the inverse function theorem implies
that $\nabla^{\circ}\Pf$ is differentiable at $y$
and that $ \nabla \nabla^{\circ} f^{\circ}(y) \left(\nabla
\nabla^{\circ}f(x)\right)^T = Id$. 
Thus, $y=\pg f(x)$ is differentiable in $x$,
and the first identity of \eqref{eq-diffpfofx} gives
\[  
\nabla \nabla^{\circ} f(x) = \frac{1}{f(x)^2}(\nabla^2
f^{\circ}(y))^{-1}\left[ f(x)I - x  \grad f(x)\right].
\]
which, after simplification,  
proves \eqref{HessianGeneralPolarityEq}.
Similarly, $f(x)I - x \grad f (x)$ is invertible. 
We re-write, using Lemma \ref{SecondPgLemma}, 
\eqref{HessianGeneralPolarityEq}
as
\[  f(x)f^{\circ}(y)(\nabla^2 f^{\circ})(y)= \left(  I - \frac{x}{f(x)}\frac{y}{f^{\circ}(y)}\right) (\nabla^2
f(x))^{-1}\left( I - (\frac{x}{f(x)}\frac{y}{f^{\circ}(y)})^T
\right) .
\]

We invert as above, using that $
{\langle x, y\rangle - f(x)f^{\circ}(y)}
=
1.
$, to get
\[ (I -  \frac{x}{f(x)}\frac{y}{f^{\circ}(y)})^{-1} 
=I - {xy},
\]
since

Thus, \eqref{eq:takedetlater} follows.
Finally,  $\det (I -  x y)=1- \langle x, y\rangle= -f(x)f^{\circ}(y).$
Thus,
\eqref{HessianDetPolarityEq} follows by
taking determinants in \eqref{eq:takedetlater}.
\end{proof}

One may readily derive  similar formulas relating the Hessians
of $f$ at $x$ and $\Jf$ at $x/f(x)$ under appropriate regularity
assumptions, for example:
$$
\nabla^2 \Jf ({\frac x{f(x)}})
=
(\nabla^2 f^{\circ})^{-1}(y)
=
f(x)\Pf(y)(I-xy)^T\nabla^2 f(x)(I-xy),
$$
where $\nabla^{\circ}f(x) = y$.
We omit the calculations.

\section{Variation of polarity}
\label{PolarityVariationSec}

In this section we consider one-parameter families $\{f_t(x)\}_{t\in\RR}$ of
convex functions.

\subsection{First variation}

  The well-known first variation formula for the Legendre
transform is:

\begin{prop}
\label
{FirstVarLProp}
Let $u(t,x)\in C^2(\RR\times\RR^n)$ with $u_t(\,\cdot\,) = u(t,\,\cdot\,)\in {\cal S}_2$
for each $t$.
Denote by $w(t,y)$ the
Legendre transforms of $u_t(x) = u(t,x)$ in the space variable, that
is,
$w(t,y) = \sup_{x\in \R^n} \left[ \langle y, x \rangle -
u(t,x)\right].$ Then,
\begin{equation}
\label{LegendreFirstVariation} \frac{\partial w}{\partial t}{(t, y)}
= -\frac{\partial u}{\partial t}(t,(\nabla u_t)^{-1}(y)).
\end{equation}
\end{prop}

For the proof, take a variation in $t$ of \eqref{LegendreExactFormulaEq},
with $x=x(t,y)=(\nabla_x u)^{-1}(y)$,
$$
\frac{\pa w}{\pa t}\Big|_y
=
\frac{dw}{dt}\Big|_y
=
\sum_{j=1}^n y_j\frac{\pa x_j}{\pa t}
-\frac{\pa u}{\pa t}\Big|_{x}
-\sum_{j=1}^n\frac{\pa u}{\pa x_j}\frac{\pa x_j}{\pa t}
=-\frac{\pa u}{\pa t}\Big|_{x},
$$
since $\nabla u(t,x)=y$.

The corresponding result for polarity is the following.\footnote{We shall use $\nabla$ and $\nabla^{\circ}$ to denote differentiation and polar differentiation with respect to the space variables.}

\begin{prop}{\rm (First variation of polarity) }\label{FirstVariationPProp}
Let $u(t,x)\in C^2(\RR\times(\RR^n\setminus \{0\}))$ with $u_t(\,\cdot\,) = u(t,\,\cdot\,)\in {\cal S}_2$ for each $t$.
Denote by $w_t=w(t,\,\cdot\,)=u_t^\circ$ the fiberwise polar.
Then for any $t$, and any $y$ such that $u_t$ is not linear on $[0,y]$,  
$$
\left(\frac{\partial }{\partial t}\log w\right){(t,y)} =
-\left(\frac{\partial}{\partial t}\log u\right){(t,\nabla^{\circ} w(t,y))}.
$$
\end{prop}

\begin{proof}
Let $y\in\Im \del^\circ u_t$. 
Proposition \ref{prop:Pfisgood} implies that 
for every $t$ and every $y$ such that $w_t$ is not linear on $[0,y]$, $w_t$ is polar differentiable at $y$. 
Denote $x(t,y)= \nabla^{\circ} w(t,y)$.
Since $u_t$ is differentiable, by Lemma \ref{SecondPgLemma} 
$y=\nabla^\circ u_t(x)$, and 
\[
w(t,y)^{-1}= {\langle x(t,y), \nabla
u(t,x(t,y))\rangle-u(t,x(t,y))}.
\]
Differentiating with
respect to $t$ gives
\begin{equation}
\begin{aligned}
\label{FirstVarFirstEq}
-\frac{1}{w^2(t,y)} \frac{\partial}{\partial t}w(t,y)
= &
-
\frac{\partial u}{\partial t} (t, x(t,y))
+\Big\langle x(t,y), (\grad \frac{\partial}{\partial t}u)(t,
x(t,y))\Big\rangle \\& +   \Big\langle x(t,y),\nabla^2 u (t,x(t,y))
\frac{\partial x}{\partial t}(t,y)\Big\rangle
\end{aligned}
\end{equation}
(where two terms have cancelled). 
By Lemma \ref{lem:moment-map-gencase}, 
$x(t,y) = (\nabla u_t)^{-1}(y/w(t,y))$, therefore
\begin{eqnarray*}
\frac{\partial}{\partial t}{x}(t,y)
&=&
(\frac{\partial}{\partial
t} (\grad u_t)^{-1}) \left( \frac{y}{w(t,y)}\right) + 
( \nabla_y (\grad u_t)^{-1})(y /w(t,y))
\frac{\partial}{\partial t} \left(\frac{y}{w(t,y)}\right).
\end{eqnarray*}
By the chain rule, $ (\nabla_y ( \grad
u_t)^{-1}) (y /w(t,y)) = (\nabla^2 \, u_t)^{-1}
(x(t,y))$.
Thus the last term in (\ref{FirstVarFirstEq}) equals
\[ \Big\langle x(t,y), \nabla^2 (u(t,x(t,y))) (\frac{\partial}{\partial
t} (\grad u_t)^{-1}) \left(\frac{y}{w(t,y)}\right) \Big\rangle + \Big\langle x(t,y),
\frac{\partial}{\partial t} \left(\frac{y}{w(t,y)}\right) \Big\rangle. \]
Plugging everything back into the original equation yields
\begin{eqnarray*}
-\frac{1}{w^2(t,y)} \frac{\partial}{\partial t}w(t,y) & = &
\Big\langle x(t,y), (\grad \frac{\partial u}{\partial t}) (t,
x(t,y))\Big\rangle \\ &&\quad+\;
\Big\langle x(t,y), (\nabla^2 u(t,x(t,y))) (\frac{\partial}{\partial t}
(\grad u_t)^{-1}) (\frac{y}{w(t,y)}) \Big\rangle \\
&&\quad-\;
\Big\langle x(t,y),   \frac{y}{w^2(t,y)} \frac{\partial w}{\partial t}(t,y) \Big\rangle
- \frac{\partial u}{\partial t} (t, x(t,y)),
\end{eqnarray*}
or, 
\begin{equation}
\label
{OneMoreFirstVarEq}
\begin{aligned}
\frac{\langle x(t,y), y\rangle -1}{w^2(t,y)}
\frac{\partial}{\partial t}w(t,y) & = \Big\langle x(t,y),
(\grad\frac{\partial u}{\partial t}) (t, x(t,y))\Big\rangle -
\frac{\partial {u}}{\partial t}(t, x(t,y)) \\ &\quad+\;
\Big\langle x(t,y), \nabla^2 u(t,x(t,y)) \left(\frac{\partial}{\partial t}
(\grad u_t)^{-1}\right) (\frac{y}{w(t,y)}) \Big\rangle.
\end{aligned}
\end{equation}
Since $(\grad u_t)^{-1} (\grad u_t (x))= x$ 
\[ \left(\frac{\partial}{\partial t} (\grad u_t)^{-1}\right)( \grad u_t (x)) +
 \iprod{\nabla(\grad u_t)^{-1}(\grad
u_t (x))}{ \frac{\partial}{\partial t} \left(\grad u_t\right) }      =
0\] or,
\[ (\frac{\partial}{\partial t} (\grad u_t)^{-1})( \grad u_t (x)) =
-  \nabla^2\, u_t (x))^{-1} \frac{\partial}{\partial t} \grad u_t
(x).
\]
Using that 
$\grad u_t (x(t,y)) = \frac{y}{w(t,y)}$, 
the first and third term on the right hand side of \eqref{OneMoreFirstVarEq}
cancel.
The result now follows %from \eqref{OneMoreFirstVarEq}
since
$\langle x(t,y), y\rangle -1 = u(t,x(t,y))w(t,y)$.
\end{proof}

One may readily combine Propositions
\ref{FirstVarLProp} and \ref{FirstVariationPProp} to get a similar formula for the first variation of $\Jf$. Under the appropriate regularity condition it reads
$$
\frac{\partial w}{\partial t} (t,x/{u_t(x)})
=
\frac{1}{u_t(x)u_t^{\star}(\nabla u_t(x))}\frac{\partial u}{\partial t}(t,x),
$$
where $w(t,y) = u_t^{\star\circ}(y)$.

\subsection{Second variation}

We recall the well-known formula for the second variation of the Legendre transform. 
Its proof follows immediately upon differentiating \eqref{LegendreFirstVariation}
(see, e.g., \cite[p. 87]{R}).

\begin{prop}
\label {LegendreSecondVarProp} 
Let $u(t,x)\in C^2(\RR\times\RR^n)$ with $u_t(\,\cdot\,) = u(t,\,\cdot\,)\in {\cal S}_2$ for each $t$.
 Let $w(t,y)=u_t^\star(y)$.
Then
\begin{equation}
\label{LegendreSecondVarOneEq} \frac{\partial^2 w}{\partial t^2}{(t,
y)} = -\frac{\partial^2 u}{\partial t^2}{(t, (\nabla u_t)^{-1}(y))}
-
\left\langle \nabla \frac{\partial u}{\partial t}(t,{(\nabla
u_t)^{-1}(y)}), \nabla \frac{\partial w}{\partial t}(t,y)\right\rangle,
\end{equation}
or equivalently
\begin{equation}
\label{LegendreSecondVarTwoEq}\frac{\partial^2 w}{\partial t^2}{(t,
y)} = -\frac{\partial^2 u}{\partial t^2}  + \langle \nabla
\frac{\partial u}{\partial t}, (\nabla^2 u_t )^{-1}\nabla
\frac{\partial u}{\partial t} \rangle
\end{equation}
where the right hand side is evaluated at $(t,(\nabla
u_t)^{-1}(y))$.
\end{prop}

For polarity we have:

\begin{thm} {\rm (Second variation of polarity)}
\label{SecondVariationPolarityThm}
Let $u(t,x)\in C^2(\RR\times(\RR^n\setminus \{0\}))$ with $u_t(\,\cdot\,) = u(t,\,\cdot\,)\in {\cal S}_2$
for each $t$.
Denote by $w_t=w(t,\,\cdot\,)=u_t^\circ$ the fiberwise polar.
Then for every $t$ and $y$ such that $w_t$ is not linear on $[0,y]$ we have  
\begin{equation}
\label
{logNoMAMatrixEq}
\left(\frac{\partial^2 }{\partial t^2}\log w\right){(t,y)}   =
-\left(\frac{\partial^2}{\partial t^2}\log u\right){ } + u \left\langle
\nabla \frac{\partial}{\partial t}(\log u) , (\nabla^2 u)^{-1} \nabla
\frac{\partial }{\partial t}(\log u)  \right\rangle, 
\end{equation}
where the right hand side is evaluated at $(t,\nabla^{\circ}{w_t}(y))$.
Equivalently,
\begin{equation}
\label
{logMAMatrixEq}
\frac{\ddot w}w
\Big|_{(t,y)}
=
-\frac1u\det
\left(
\begin{array}{cccc}
-u^2\ddot{(1/u)}  &   & u\nabla(\dot u/u) &   \\
%                            &   &  &   \\
(u\nabla(\dot u/u))^T &   & \nabla^2_x u  &  \\
%                            &   &   &   \\
\end{array}
\right)\bigg|_{(t,\pg w(t)(y))}.
\end{equation}
\end{thm}

\begin{proof}
We differentiate the first variation formula (Proposition
\ref{FirstVariationPProp})
\[
\left(\frac{\partial }{\partial t}\log w\right){(t,y)} =
-\left(\frac{\partial}{\partial t}\log u\right){(t,x(t,y))}
\]
with respect to $t$, where again $x(t,y):= \nabla^{\circ}(w(t))$.
We obtain that
\[
\left(\frac{\partial^2 }{\partial t^2}\log w\right){(t,y)}   =
-\left(\frac{\partial^2}{\partial t^2}\log u\right){(t, x(t,y)) }
- 
\Big\langle
\grad( \frac{\partial}{\partial t}\log u)(t,x(t,y)),
\frac{\partial }{\partial t}{x}(t,y)\Big\rangle.
\]
By the computations in the proof of Proposition \ref{FirstVariationPProp}
we have that
\begin{eqnarray*}\frac{\partial}{\partial t}{x}(t,y) &=&    
- (\nabla^2 \, u_t)^{-1}(x(t,y)) \left[\frac{\partial}{\partial t}{\nabla u(t,x(t,y))} +
\frac{y}{w^2(t,y)}\frac{\partial}{\partial t}{w(t,y)}\right]
 \\
& = & -  (\nabla^2 \, u_t)^{-1}
(x(t,y)) u(t, x(t,y))\left[\frac{\frac{\partial}{\partial t}{\nabla
u(t,x(t,y))}}{u(t,x(t,y))} - \frac{\nabla u (t,
x(t,y))}{u^2(t,x(t,y))}\frac{\partial u}{\partial
t}{(t,x(t,y))}\right] 
\\
& = & -  u (\nabla^2 \, u)^{-1} 
  \left[\left({\frac{\partial}{\partial t} {\nabla (\log u)}}\right)(t,x(t,y)) \right].
\end{eqnarray*}
Plugging into the formula above we get that
\[
\left(\frac{\partial^2 }{\partial t^2}\log w\right){(t,y)}   =
\left[-\left(\frac{\partial^2}{\partial t^2}\log u\right){  } + u \left\langle
  (\nabla^2 \, u )^{-1}
 {\frac{\partial}{\partial t} {\nabla (\log u)}},\grad( \frac{\partial}{\partial t}\log u)
 \right\rangle\right]_{(t, x(t,y))},
\]
proving \eqref{logNoMAMatrixEq}.
Equation \eqref{logMAMatrixEq} follows from this
and Proposition \ref{FirstVariationPProp}.
\end{proof}

\def\Lu{{u^\star}}

\def\Pu{{u^\circ}}

\begin{rem}
{\rm
Note that the last term can be expressed more symmetrically as follows:
$$\begin{aligned}
u  \left\langle
\nabla \frac{\partial}{\partial t}(\log u) , (\nabla^2 u)^{-1} \nabla
\frac{\partial }{\partial t}(\log u)  \right\rangle
=
uw \left\langle
\nabla \frac{\partial}{\partial t}(\log u) ,  \nabla
\frac{\partial }{\partial t}(\log w) (I-(\nabla u)^T\cdot \nabla w)^{-1} \right\rangle.
\end{aligned}$$
We omit the computations.}
\end{rem}

\section{First order equations}
\label{FirstOrderSec}

The first order analysis
enables us to linearize a family of first order PDEs, analogous to
the linearizaton of  the Hamilton--Jacobi equation by the Legendre transform.
Define the operation $\boxdot$ by
$$
 a\boxdot b=
(a^{\circ}+b^{\circ})^{\circ}.
$$
This is shown to be a sort of
geometric inf-convolution in \S\ref{GinfSec}
where a precise formula is derived.

\begin{thm}\label{FirstOrderThm}
Let $g\in \Cvx_0(\RR^n)\cap C^2(\RR^n \setminus \{0\})$.
Let $f\in {\cal S}_2$ and non-linear at infinity.  
Then the equation 
\begin{equation}
\label{FirstEqOne}
\frac{1}{u}\frac{\partial u}{\partial t}+{u_t^{\star}(\nabla u)}g\bigg(\frac{\nabla u}{u_t^{\star}(\nabla u)}\bigg)=0,\quad
u(0,x) = f(x),
\end{equation}
admits a unique non-linear at infinity solution $u(t)\in{\cal S}_2$ given by
\begin{equation}
\label{FirstEqTwo}
 u(t,x) =  f\boxdot \frac{1}{t}g^{\circ}.
\end{equation}
In particular, there exists a solution for all time
$t\ge 0$.
\end{thm}
A similar result holds for the Dirichlet problem
with convex data.

\begin{proof}
Note that by Lemma \ref{lem:linzone}  the function $f^{\circ}$ 
has an empty linearity zone, and therefore so does $f^{\circ}+tg$ for any $t\ge 0$. 
We may thus apply Proposition \ref{FirstVariationPProp}, which implies that the function $(f^{\circ}+tg)^{\circ}$ satisfies our original equation.
\end{proof}

As an application, an equation reminiscent of Burgers' equation,
\begin{eqnarray*}
&&\frac{\partial u}{\partial t}(t,x)
+\|\nabla u(t,x)\|{u(t,x)} = 0,\qquad u(0,x) = f(x),
\end{eqnarray*}
can be solved for all $t\ge0$, with
\[
u(t)\equiv u(t,\,\cdot\,) = f\boxdot \frac{1}{t}\|\cdot\|^{*}
= (f^{\circ}+t\|\cdot\|)^{\circ},
\]
where the polarity operation is performed with respect to the
variable $x$ only. 
Here $\|\cdot\|^{*}$ denotes the norm dual to $\|\cdot\|$.
If $f$ is a norm then so is $u(t)$ for each $t$.

Similarly, a solution of 
\begin{eqnarray*}
\frac{\partial}{\del t} \log u(t,x)
+\frac{1}{2}\frac{|\nabla
u|^2}{ u_t^\star(\nabla u)}=0,
\qquad u(0,x) = f(x),
\end{eqnarray*}
is given by
\[ u(t)\equiv u(t,\,\cdot\,) = f\boxdot \frac{1}{2t}\|\cdot\|^{2}
= (f^{\circ}+\frac{t}{2}|\cdot|^2)^{\circ}.
\]

We remark that analogously there are PDEs of first and second
order linearized by the transform $\P\circ\L$, and we omit the calculations
for brevity.

\section{The polar \MA equation}
\label{SecondOrderSec}

To put our results in this section in perspective,
it is good to keep in mind the classical result that the partial Legendre
transform linearizes the homogeneous real \MA (HRMA), written schematically
as  $\det\nabla^2 f=0$ \cite{S}.
This is contained in Proposition \ref{LegendreSecondVarProp}:
if $u(t)\in\SCvx\cap C^\infty$ for each $t$ then
$\det\nabla^2_{t,x} u=\det\nabla^2_xu(\ddot u-|\nabla\dot u|^2_{(\nabla^2 u)^{-1}})=0$
if and only if $\ddot u(t)^\star=0$, where $u(t)^\star$ denotes the Legendre
transform of $u(t,\,\cdot\,)$ in the $x$ variables.

The following is a consequence of Theorem 
\ref{SecondVariationPolarityThm}.
We denote 
$|X|^2_g:=g(X,X)$ for any semi-Riemannian metric $g$.

\begin{thm}
Let $u_0,u_T\in {\cal S}_2$ and non-linear at infinity.  
The Dirichlet problem 
\begin{equation}\label
{LogMADirichletEq}
\ddot{(1/u)}
+
|\nabla(\dot u/u)|^2_{(\nabla^2 u)^{-1}}=0,
\qquad
u(0,\,\cdot\,)=u_0,\quad
u(T,\,\cdot\,)=u_T,
\end{equation}
admits a unique 
non-linear at infinity solution $u(t)\in{\cal S}_2$ given by
\begin{equation}\label{eq:sol-dirlogMA}
u(t,x) = \left(\left(1-\frac tT\right)u^\circ_0+\frac tTu^\circ_1\right)^\circ(x)=
\left(\frac{Tu_0}{T-t} \boxdot \frac{Tu_1}{t}\right)(x).%,\qquad
\end{equation}
\end{thm}

We call \eqref{LogMADirichletEq} the homogeneous
polar \MA equation. 
Similarly, the solution
to the Cauchy problem for the 
homogeneous polar \MA equation follows by combining
Proposition \ref{FirstVariationPProp} and
Theorem 
\ref{SecondVariationPolarityThm}.

\begin{thm}
Let $u_0\in {\cal S}_2$ and non-linear at infinity, and let $\dot{u}_0   \in C^2(\RR^n)$ satisfy $\dot u_0(0)=0$. 
The Cauchy problem 
\begin{equation}\label
{LogMACauchyEq}
\ddot{(1/u)}
+
|\nabla(\dot u/u)|^2_{(\nabla^2 u)^{-1}}=0,
\qquad
u(0,\,\cdot\,)=u_0,\quad
\dot{u}(0,\,\cdot\,)=\dot{u}_0,
\end{equation}
admits a unique 
non-linear at infinity solution $u(t)\in{\cal S}_2$ given by 
\begin{equation}\label{eq:sol-cauchy-LogMA}
u(t,x) = \Big(u^\circ_0\cdot(1-tv) \Big)^\circ(x),
\qquad t\in[0,T),
\end{equation}
where $v(0)=0$ and 
\[ v(y) = \frac{\dot{u_0}(\pg u_0^{\circ}(y))}{u_0(\pg u_0^{\circ}(y))},
\qquad y\not=0,\]
and where $t\in [0,T)$ with $T=T(u_0,\dot u_0)$ the supremum
over all $t>0$ such that the function
 $u^\circ_0\cdot(1-tv)\in {\cal S}_2$ and is nonlinear at infinity. 
\end{thm}

\section{Geometric inf-convolution}\label{sec:Ginf}

\label{GinfSec}
In this section we derive an explicit formula for the
the  solutions to the PDEs 
presented in the preceding sections.
We refer to 
$$
f\boxdot g = (f^\circ + g^\circ)^\circ.
$$
as {\it geometric inf-convolution} of $f$ and $g$.
The next lemma justifies this name.
It gives a formula for such as expression,
reminiscent of the formula for inf-convolution
\cite[p. 33]{Ro}
$$
f\square g (x) = \inf_{y+z = x} \left ( f(y)
+ g(z) \right) = 
(f^\star + g^\star)^\star.
$$

\begin{lem}\label{lem-ginf}
For $f,g \in \Cvx_0(\R^n)$ that vanish only at the origin,
$f \boxdot g\in\Cvx_0(\RR^n)$ and
\[ (f \boxdot g)(x) = \inf \left\{ \left(f(y)^{-1} +
g(z)^{-1}\right)^{-1} : \frac{x - y}{f(y)} = \frac{z-x}{g(z)}\right\}.\]
\end{lem}

\begin{proof}[Proof of Lemma \ref{lem-ginf}]

We will solve the equation in $h$, $f^\circ + g^\circ = h^\circ$. Then it
will be left to check that $h$ is geometric convex.
\begin{eqnarray*}
(f^\circ + g^\circ)(x) &=& \sup_{y,z\in \R^n} \left( \frac{\langle x,y
\rangle - 1 }{f(y)} + \frac{\langle x,z \rangle - 1 }{g(z)}
\right)
\\
& = & \sup_{y,z\in \R^n} \left( \frac {\langle x, \frac{g(z)y +
f(y)z}{(f(y) + g(z))}\rangle  - 1 } {f(y)g(z)(f(y) + g(z))^{-1}}
\right)\\
& = & \sup_{w\in \R^n} \left( \frac {\langle x, w\rangle - 1
}{\inf_{\{y,z\in \R^n: w = \frac{g(z)y + f(y)z}{(f(y) + g(z))} \}}
(f(y)^{-1} + g(z)^{-1})^{-1}}
\right) \\
\end{eqnarray*}
Letting
\[
h(w) = \inf \{ (f(y)^{-1} + g(z)^{-1})^{-1} : y,z\in
\R^n, ~~ w = \frac{g(z)y + f(y)z}{f(y) + g(z)} \} ,
\]
we see that
the last expression is $h^\circ(x)$. Then rearrange $w = \frac{g(z)y
+ f(y)z}{f(y) + g(z)}$ as $\frac{w - y}{f(y)} =
\frac{z-w}{g(z)}$ or $ {(w - y)}{g(z)} =  {(z-w)}{f(y)}$ .

It remains to verify that the resulting
function is geometric convex.
Denote
\[ K_{\varphi} = \overline{\{(x, y)\in \R^n\times \R^+: y\varphi(x/y)\le
1\}}.  \]
Then \cite{AM2}
\[ K_{f\boxdot g} = \overline{K_f + K_g}. \]
Write
\begin{eqnarray*}
K_f + K_g & = & \{ (z,y): x = z_1 + z_2, y = y_1 + y_2,
y_1f(z_1/y_1) \le 1,   y_2g(z_2/y_2) \le 1\}\\
& = & \{ (xy,y): x = \frac{x_1y_1 + x_2y_2}{y_1 + y_2}, y = y_1 +
y_2,
y_1f(x_1) \le 1, y_2g(x_2) \le 1\}.\\
\end{eqnarray*}
Thus,
$K_h  = \overline{K_f + K_g}$, and
\[ h(x) = \| (x,1)\|_{K_h}  = \inf\{ 1/y : (xy,y) \in
K_h\} = \inf\{ 1/y : (xy,y) \in K_f + K_g\} . \]
Therefore
\begin{eqnarray}
( f\boxdot g )(x)  & = & \inf \{ \frac{1}{y_1 + y_2} : x =
\frac{x_1y_1 + x_2y_2}{y_1 + y_2}, y_1f(x_1) \le 1, y_2g(x_2) \le
1\}.
\end{eqnarray}
In the strictly convex case, the boundary of $K_h$ is a subset of the
closure of the Minkowski sum of the boundaries of the sets $K_f$ and $K_g$, which
means that we can without loss of generality assume in
the infimum above $y_1 = 1/f(x_1)$ and $y_2 = 1/f(x_2)$. We end up
with
\begin{eqnarray*}
( f\boxdot g )(x)  & = & \inf \{ \frac{1}{f(x_1)^{-1} + f(x_2)^{-1}}
: x = \frac{x_1(f(x_1))^{-1} + x_2(f(x_2))^{-1}}{(f(x_1))^{-1} +
(f(x_2))^{-1}}\}.
\end{eqnarray*}
Rearranging, the result follows.
\end{proof}

Next we present a formula for the polar gradient of the function $f+ g$ at a point $x$.

\begin{lem}
Let $f,g\in \Cvx_0{\RR^n}$ with $\dom(f)=\dom(g)=\RR^n$. Assume both are polar differeitable at some $x\in\RR^n$. Then $f+g$ is polar differentiable at $x$ and  
\[ \pg (f+g)(x)=\left(\frac{g^{\circ}(\pg g(x))}{\Pf (\pg f (x))+g^{\circ}(\pg g(x))}\right) \pg f (x)
+\left(\frac{f^{\circ}(\pg f(x))}{\Pf (\pg f (x))+g^{\circ}(\pg g(x))}\right) \pg g(x). \]
\end{lem}
 Note that $\pg (tf) (x)= \pg (f)(x)$ by the definition of the polar gradient. Thus we get the formula
 \[ \pg (f+tg) = 
\left(\frac{g^{\circ}(\pg g(x))}{t\Pf (\pg f (x))+g^{\circ}(\pg g(x))}\right) \pg f (x)
+\left(\frac{tf^{\circ}(\pg f(x))}{t\Pf (\pg f (x))+g^{\circ}(\pg g(x))}\right) \pg g(x). \]
 
% \]

\begin{proof}
As $f$ and $g$ are polar differentiable at $x$, by Lemma \ref{SecondPgLemma} they are both differentiable at $x$, hence so is $f+g$.
Denote $y_1 = \pg f (x)$, $y_2 = \pg g (x)$ and 
\[ z = y_1\frac{g^{\circ}(y_2)}{\Pf (y_1)+g^{\circ}(y_2)} + {y_2}\frac{f^{\circ}(y_1)}{\Pf (y_1)+g^{\circ}(y_2)}.\] 
We shall show that $z\in \psg (f+g)(x)$, and by Lemma \ref{SecondPgLemma} once again, get that $z = \pg (f+g)(x)$, as needed. 
By definition  
\[ f(x)\Pf (y_1) = \iprod{x}{y_1}-1, \qquad 
g(x)g^{\circ} (y_2) = \iprod{x}{y_2}-1.
\]
Since $x\in {\rm int}(\dom(f)\cap \dom(g)) = \RR^n$ we have that $\Pf (y_1)g^{\circ} (y_2)\neq 0$. Thus
\[f(x)+g(x) = 
\iprod{x}{\frac{y_1}{\Pf (y_1)} + \frac{y_2}{g^{\circ}(y_2)}}-\left( \frac{1}{\Pf(y_1)}+\frac{1}{g^{\circ}(y_2)}\right).
\]
Rearrange to get 
\[(f(x)+g(x))\left( {\Pf(y_1)}^{-1}+{g^{\circ}(y_2)}^{-1}\right)^{-1} = 
\iprod{x}{z}-1.
\]
Using Lemma \ref{lem-ginf} we have that 
\[ (f+g)^{\circ}(z) \le (f^{\circ}(y_1)^{-1}+g^{\circ}(y_2)^{-1})^{-1}\]
so that 
\[(f+g)(x)(f+g)^{\circ}(z)\le  
\iprod{x}{z}-1.
\]
The opposite inequality holds by the definition of $(fg)^{\circ}$, so we get equality, which means $z\in \psg(f+g)(x)$, as claimed. 
\end{proof}
Note that in the proof we obtained the formula
\[ (f+g)^{\circ}((1-\lambda) \pg f (x)
+\lambda \pg g(x)) =  (f^{\circ}(\pg f (x))^{-1}+g^{\circ}(\pg g(x))^{-1})^{-1}\]
for $\lambda = \frac{f^{\circ}(\pg f(x))}{\Pf (\pg f (x))+g^{\circ}(\pg g(x))}$, which is similar to a corresponding formula for inf-convolution.

\subsection*{Acknowledgements}

This research was supported by BSF grant 2012236. SAA was supported by 
ISF grant 247/11, and YAR by NSF grants DMS-0802923,1206284 and by a Sloan Research Fellowship.

\bigskip

\begin{spacing}{0}

\end{spacing}

\bigskip

{\sc Tel-Aviv University}

{\tt shiri@post.tau.ac.il}

\medskip

{\sc University of Maryland}

{\tt yanir@umd.edu}

\end{document}